\documentclass{article}
\usepackage[a4paper, total={6.2in, 8.2in}]{geometry}
\usepackage[utf8]{inputenc}
\usepackage{mathtools}
\usepackage{amsmath}
\usepackage{amsthm}
\usepackage{amssymb}
\usepackage{booktabs}
\usepackage{adjustbox}
\usepackage{ltablex,array}
\usepackage{longtable}
\usepackage{makecell}
\usepackage{marvosym}

\let\svthefootnote\thefootnote
\newcommand\freefootnote[1]{%
  \let\thefootnote\relax%
  \footnotetext{#1}%
  \let\thefootnote\svthefootnote%
}

\newcommand\restr[2]{{%
\left.
\kern-
\nulldelimiterspace %
#1 %
\right|_{#2} %
}}
\usepackage{tikz}
\usepackage{tikz-cd}
\tikzcdset{scale cd/.style={every label/.append style={scale=#1},
    cells={nodes={scale=#1}}}}
\usepackage{quiver}
\usetikzlibrary{positioning} \usetikzlibrary{decorations.pathreplacing}
\usetikzlibrary{arrows}
\tikzset{>=stealth}
\usepackage{xcolor}
\newtheorem{theorem}{Theorem}
\newtheorem{definition}[theorem]{Definition}

\newtheorem{proposition}[theorem]{Proposition}

\newtheorem{corollary}[theorem]{Corollary}
\newtheorem{question}[theorem]{Question}
\theoremstyle{definition}

\newtheorem{remark}[theorem]{Remark}

\theoremstyle{definition}

\usepackage[
backend=biber,
style=alphabetic,
sorting=nyt
]{biblatex}
\usepackage{mathtools}
\usepackage{hyperref}
\hypersetup{
    breaklinks=true,
    }
\usepackage[hyphenbreaks]{xurl}

\begin{document}
\title{Pure maps are strict monomorphisms}
\author{Kristóf Kanalas}
\date{}
\maketitle

\begin{abstract}
    We prove that $i)$ if $\mathcal{A}$ is $\lambda $-accessible and it is axiomatizable in (finitary) coherent logic then $\lambda $-pure maps are strict monomorphisms and $ii)$ if there is a proper class of strongly compact cardinals and $\mathcal{A}$ is $\lambda $-accessible then for some $\mu \vartriangleright \lambda $ every $\mu $-pure map is a strict monomorphism. \freefootnote{Supported by the Grant Agency of the Czech Republic (grant
22-02964S), and by Masaryk University (project MUNI/A/1457/2023).}
\end{abstract}

\subsubsection*{Introduction} In the book \cite{rosicky} the authors ask if it is true that in a $\lambda $-accessible category $\lambda $-pure maps are regular monomorphisms. In \cite{ADAMEK19961} a counterexample is given: there is a small $\omega $-accessible category with an $\omega $-pure map which is an epi (but not iso), hence it is not even a strong monomorphism.

Some positive results are known: in accessible categories with pushouts (\cite{ADAMEK19961}) or in accessible categories with products (\cite{hupreprint}, \cite{HU1997163}) the answer is affirmative. Here we will give some further positive results.

Our argument goes through the syntactic side of categorical logic: $(\lambda ,\kappa )$-coherent categories, $\kappa $-sites, $\kappa $-toposes, etc. For these notions we refer to \cite[Section 2]{bvalued} or alternatively to \cite{presheaftype}. To simplify notation $(\kappa ,\kappa )$-coherent categories (and functors) will be called $\kappa $-coherent.

The main idea is the following: given a $\lambda $-accessible category $\mathcal{A}$ we may axiomatize it by a so-called $\lambda $-site: there is a small $\lambda $-lex category $\mathcal{C}$ with some specified arrow-families (wide cospans), whose set we call $E$, such that $\mathcal{A}$ is equivalent to the full subcategory of $\mathbf{Lex}_{\lambda }(\mathcal{C},\mathbf{Set})$ spanned by those $\lambda $-lex functors which send the $E$-families to jointly surjective ones.

If $\mathcal{C}$ is regular then every $\lambda $-pure map $\alpha :M\to N$ in $\mathcal{A}$ is a regular monomorphism in $\mathbf{Lex}_{\lambda }(\mathcal{C},\mathbf{Set})$. So it is the equalizer of some pair $\beta _1, \beta _2:N\to F$ where $F$ is a  $\lambda $-lex functor. Therefore it would be enough to find a monomorphism from $F$ to some model. This can almost be done: if $(\mathcal{C},E)$ satisfies some further exactness properties then every $\lambda $-lex functor admits a regular monomorphism to a product of models, and hence $\alpha $ is the joint equalizer of a set of parallel pairs. Such maps are called strict monomorphisms. 

In what follows we will clarify these claims.

\begin{definition}
    $\mathcal{C}$, $\mathcal{D}$,  $M,N:\mathcal{C}\to \mathcal{D}$ are lex. $\alpha :M\Rightarrow N $ is elementary if the naturality squares at monos are pullbacks.
\end{definition} 

\begin{remark}
    Think of $\mathcal{C}$ as a syntactic category (in some fragment of first-order logic), and $M,N:\mathcal{C}\to \mathbf{Set}$ as models. Then a map $M\Rightarrow N$ is elementary iff the squares
\[
\adjustbox{scale=0.9}{
\begin{tikzcd}
	{[\varphi (\vec{x})]^M} && {[\vec{x}=\vec{x}]^M} \\
	\\
	{[\varphi (\vec{x})]^N} && {[\vec{x}=\vec{x}]^N}
	\arrow[hook, from=1-1, to=1-3]
	\arrow[from=1-1, to=3-1]
	\arrow[from=1-3, to=3-3]
	\arrow[hook, from=3-1, to=3-3]
\end{tikzcd}
}
\]
    are pullbacks. So maps between models are functions preserving the formulas in the given fragment (commutativity), elementary maps are functions preserving and reflecting them (pullback). When $\mathcal{C}$ is a coherent category (i.e.~formulas in $\mathcal{C}$ are positive existential), maps are just homomorphisms, elementary maps are those homomorphisms which reflect positive existential formulas. These are called immersions in positive logic. 
\end{remark}

\begin{remark}
    Elementary maps were first defined in \cite{barr}. They are also called elementary in \cite{lurie}. They are called immersions in \cite{exclosedkamsma}. They are called mono-cartesian in \cite{GARNER2020102831}.
\end{remark}

\begin{proposition}
    elementary $\Rightarrow $ pointwise mono.
\end{proposition}

\begin{proof}
    \cite[Proposition 5.8]{bvalued}
\end{proof}

At regular monomorphisms the converse is also true:

\begin{proposition}
\label{atregmonopb}
     $\mathcal{C}$, $\mathcal{D}$,  $M,N:\mathcal{C}\to \mathcal{D}$ are lex, $\alpha :M\Rightarrow N $ is pointwise mono. Then at any regular monomorphism $i:u\hookrightarrow x$ of $\mathcal{C}$ the naturality square is a pullback. 
\end{proposition}

\begin{proof}
   Assume that $i$ is the equalizer of $f,g:x\to y$. Given a cospan $Gu\xleftarrow{k}Z\xrightarrow{h}Fx$ with $\alpha _x\circ h = Gi \circ k$ as in
\[\begin{tikzcd}
	Z \\
	& Fu && Fx && Fy \\
	\\
	& Gu && Gx && Gy
	\arrow["t"{description}, dashed, from=1-1, to=2-2]
	\arrow["h"{description}, curve={height=-12pt}, from=1-1, to=2-4]
	\arrow["k"{description}, curve={height=12pt}, from=1-1, to=4-2]
	\arrow["Fi"{description}, hook, from=2-2, to=2-4]
	\arrow["{\alpha _u}", hook', from=2-2, to=4-2]
	\arrow["Ff", shift left=2, from=2-4, to=2-6]
	\arrow["Fg"', shift right=2, from=2-4, to=2-6]
	\arrow["{\alpha _x}", hook', from=2-4, to=4-4]
	\arrow["{\alpha _y}", hook', from=2-6, to=4-6]
	\arrow["Gi"{description}, hook, from=4-2, to=4-4]
	\arrow["Gf", shift left=2, from=4-4, to=4-6]
	\arrow["Gg"', shift right=2, from=4-4, to=4-6]
\end{tikzcd}\]
   we see that $h$ factors through $Fi$ (via the dashed arrow $t:Z\to Fu$), since $Fi$ is the equalizer of $(\alpha _y\circ Ff,\alpha _y\circ Fg)=(Gf\circ \alpha _x, Gg\circ \alpha _x)$. We also get $\alpha _u \circ t = k$ as $Gi$ is mono.
\end{proof}

\begin{remark}
    In \cite[Corollary 5.13]{bvalued} we proved that if the unit of a geometric morphism is pointwise mono then it is elementary. As in a Grothendieck topos all monos are regular, the above proposition yields a second proof of this.
\end{remark}

\begin{remark}
\label{keyobserv}
    $\mathcal{C}$, $M,N:\mathcal{C}\to \mathbf{Set}$ are lex. By definition, $\alpha :M\Rightarrow N$ is elementary iff for every mono $u\hookrightarrow x$ in $\mathcal{C}$
\[
\adjustbox{scale=0.9}{
\begin{tikzcd}
	Mu && Mx \\
	\\
	Nu && Nx
	\arrow[hook, from=1-1, to=1-3]
	\arrow[from=1-1, to=3-1]
	\arrow["{\alpha _x}", from=1-3, to=3-3]
	\arrow[hook, from=3-1, to=3-3]
\end{tikzcd}
}
\]
is a pullback, i.e.~iff for every $a\in Mx$ if $\alpha _x(a)\in Nu$ then $a\in Mu$. In other terms every commutative square
\[
\adjustbox{scale=0.9}{
\begin{tikzcd}
	{\widehat{x}:= \ \mathcal{C}(x,-)} && M \\
	\\
	{\widehat{u}:= \ \mathcal{C}(u,-)} && N
	\arrow[""{name=0, anchor=center, inner sep=0}, from=1-1, to=1-3]
	\arrow["{\widehat{i}:= \ -\circ i}"', from=1-1, to=3-1]
	\arrow["\alpha", from=1-3, to=3-3]
	\arrow[""{name=1, anchor=center, inner sep=0}, dashed, from=3-1, to=1-3]
	\arrow[from=3-1, to=3-3]
	\arrow["{=}"', draw=none, from=0, to=1]
\end{tikzcd}
}
\]
admits a diagonal map which makes the upper triangle commutative.
\end{remark}

This looks very similar to the definition of pure maps (cf.~\cite[Definition 2.27]{rosicky}). 

\begin{definition}
    $\mathcal{K}$ is $\lambda $-accessible ($\lambda $ infinite, regular). A map $f:x\to y$ is $\lambda $-pure if for any commutative square
\[\begin{tikzcd}
	a & x \\
	b & y
	\arrow[from=1-1, to=1-2]
	\arrow[from=1-1, to=2-1]
	\arrow["f", from=1-2, to=2-2]
	\arrow[from=2-1, to=2-2]
\end{tikzcd}\]
    with $a,b$ $\lambda $-presentable, there is a diagonal map $h:b\to x$ making the upper triangle commute.
\end{definition}

\begin{proposition}
\label{purethenelem}
    $\kappa \leq \lambda $ are infinite, regular cardinals. $(\mathcal{C},E)$ is a $\kappa $-site, $Mod_{\kappa }(\mathcal{C},E)$ is $\lambda $-accessible. Then $\lambda $-pure $\Rightarrow $ elementary.
\end{proposition}

\begin{proof}
    Given a square as in Remark~\ref{keyobserv}, we can factor it through some map $M_0\to N_0$ between $\lambda $-presentable models (using that $\kappa $-filtered (and hence $\lambda $-filtered) colimits of $\kappa $-lex $E$-preserving functors are computed in $\mathbf{Set}^{\mathcal{C}}$, where representables are tiny).
\[
\adjustbox{scale=0.9}{
\begin{tikzcd}
	{\widehat{x}} && {M_0} && M \\
	& {=} \\
	{\widehat{u}} && {N_0} && N
	\arrow[from=1-1, to=1-3]
	\arrow[curve={height=-12pt}, from=1-1, to=1-5]
	\arrow[from=1-1, to=3-1]
	\arrow[""{name=0, anchor=center, inner sep=0}, from=1-3, to=1-5]
	\arrow[from=1-3, to=3-3]
	\arrow[from=1-5, to=3-5]
	\arrow[from=3-1, to=3-3]
	\arrow[curve={height=12pt}, from=3-1, to=3-5]
	\arrow[""{name=1, anchor=center, inner sep=0}, from=3-3, to=1-5]
	\arrow[from=3-3, to=3-5]
	\arrow["{=}"', draw=none, from=0, to=1]
\end{tikzcd} 
}
\]
\end{proof}

\begin{remark}
    The notion of an elementary map depends on the presentation of the given accessible category as a category of functors. For instance, the category of Abelian groups $\mathbf{Ab}$ is equivalent to $\mathbf{Lex}(\mathbf{Ab}_{f.p.}^{op},\mathbf{Set})$, and this equivalence sends $i:\mathbb{Z}\xrightarrow{\cdot 2} \mathbb{Z}$ to $i_{\circ }:\mathbf{Ab}_{f.p.}(-,\mathbb{Z})\to \mathbf{Ab}_{f.p.}(-,\mathbb{Z})$. This map is not $\omega $-pure as $i$ is not $\omega $-pure and purity is invariant under equivalence. However, it is easy to check that it is elementary.
\end{remark}

This relates $\lambda $-pure vs.~elementary. Now let's relate elementary vs.~regular mono.

We will need the following:

\begin{theorem}
\label{positselski}
    Let $\mathcal{K}$ be locally $\kappa $-presentable. Then every $\kappa $-small diagram is the $\kappa $-directed colimit of such diagrams formed among $\kappa $-presentable objects. 
\end{theorem}

\begin{proof}
    \cite[Theorem 6.1]{diaginacc}
\end{proof}

\begin{theorem}
    $\mathcal{C}$ is $\kappa $-lex, regular. Then $\mathbf{Lex}_{\kappa }(\mathcal{C},\mathbf{Set})^{op} = Pro_{\kappa }(\mathcal{C})$ is $\kappa $-lex, regular and the Yoneda-embedding $Y:\mathcal{C}\hookrightarrow \mathbf{Lex}_{\kappa }(\mathcal{C},\mathbf{Set})^{op}$ is $\kappa $-lex, regular.
\end{theorem}

\begin{proof}
    For $\kappa =\omega $ this is \cite[Theorem 1]{barr}. Alternatively: \cite[Corollary 10, Proposition 11, Remark 13 in Lecture 14X]{lurie}. We repeat the proof.

    $\mathbf{Lex}_{\kappa }(\mathcal{C},\mathbf{Set})$ is a $\kappa $-accessible category which is complete. Therefore it is locally $\kappa $-presentable, hence cocomplete. So $Pro_{\kappa }(\mathcal{C})$ is complete and cocomplete.

    We claim that $Y:\mathcal{C}\to Pro_{\kappa }(\mathcal{C})$ preserves $\kappa $-small limits and all colimits which exist in $\mathcal{C}$. The latter is clear (the Hom-functor in its first variable turns colimits into limits in $\mathbf{Set}^{\mathcal{C}}$, and those are the same as limits in $\mathbf{Lex}_{\kappa }(\mathcal{C},\mathbf{Set})$). 

    Given the $Y$-image of a $\kappa $-small limit diagram (in $\mathbf{Lex}_{\kappa }(\mathcal{C},\mathbf{Set})$) we have to show that it is a colimit, i.e.~we have to find a unique dashed arrow in

\[
\adjustbox{scale=0.9}{
\begin{tikzcd}
	& F \\
	& {\widehat{x}} \\
	{\widehat{x_i}} && {\widehat{x_j}}
	\arrow[dashed, from=2-2, to=1-2]
	\arrow[curve={height=-6pt}, from=3-1, to=1-2]
	\arrow[from=3-1, to=2-2]
	\arrow[curve={height=6pt}, from=3-3, to=1-2]
	\arrow[from=3-3, to=2-2]
	\arrow[from=3-3, to=3-1]
\end{tikzcd}
}
\]

The images of identities form a compatible family $(a_i\in Fx_i)_i$, that is an element of $Fx=lim\ Fx_i$, and $1_x\mapsto (a_i)_i$ is the unique arrow which makes the diagram commutative.

    We got that $Y:\mathcal{C}\hookrightarrow Pro_{\kappa }(\mathcal{C})$ preserves $\kappa $-small limits and effective epimorphisms. 

By Theorem~\ref{positselski} every $\kappa $-small diagram in $\mathbf{Lex}_{\kappa }(\mathcal{C},\mathbf{Set})$ arises as the $\kappa $-directed colimit of such diagrams formed among $\kappa $-presentable objects. But since $\mathcal{C}$ is regular it is Cauchy-complete (splittings of idempotents are given by image factorization), therefore it follows that retracts of representables are representables. So  every $\kappa $-small diagram is the $\kappa $-directed colimit of such diagrams formed among representables.

In $\mathbf{Lex}_{\kappa }(\mathcal{C},\mathbf{Set})$ epi - effective mono factorizations admit the following description: 

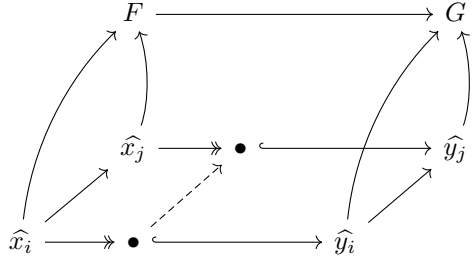
\[
\adjustbox{scale=0.9}{
\begin{tikzcd}
	& F &&& G \\
	\\
	& {\widehat{x_j}} & \bullet && {\widehat{y_j}} \\
	{\widehat{x_i}} & \bullet && {\widehat{y_i}}
	\arrow[from=1-2, to=1-5]
	\arrow[curve={height=6pt}, from=3-2, to=1-2]
	\arrow[two heads, from=3-2, to=3-3]
	\arrow[hook, from=3-3, to=3-5]
	\arrow[curve={height=6pt}, from=3-5, to=1-5]
	\arrow[curve={height=-12pt}, from=4-1, to=1-2]
	\arrow[from=4-1, to=3-2]
	\arrow[two heads, from=4-1, to=4-2]
	\arrow[dashed, from=4-2, to=3-3]
	\arrow[hook, from=4-2, to=4-4]
	\arrow[curve={height=-12pt}, from=4-4, to=1-5]
	\arrow[from=4-4, to=3-5]
\end{tikzcd}
}
\]

That is: write the map as a $\kappa $-directed colimit of maps between representables, take the epi - effective mono factorization of those (which is the $Y$-image of the effective epi - mono factorization in $\mathcal{C}$), get induced maps between the middle-terms (using effective $\Rightarrow $ strong), then take colimits (and use that in locally $\kappa $-presentable categories both epis and effective monos are closed under $\kappa $-filtered colimits, see \cite[Corollary 1.60]{rosicky}).

In particular it follows that every epi, as well as every effective mono is the $\kappa $-filtered colimit of such maps between representables.

Finally, we have to show that the pushout of an effective mono is effective mono. Given a span $F_1\Leftarrow F_0\Rightarrow F_2$ we can write it as a $\kappa $-directed colimit of spans formed among representables. We can form the epi - effective mono factorization of the first legs layer-wise, then take pushouts layer-wise.

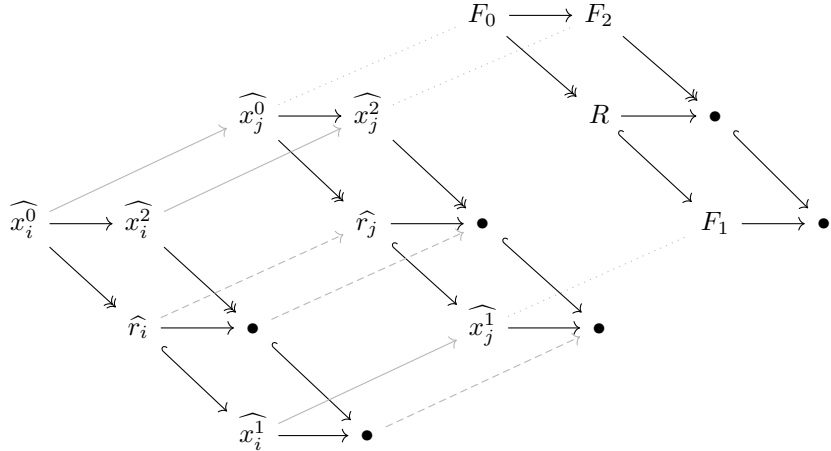
\[
\adjustbox{scale=0.9}{
\begin{tikzcd}
	&&&& {F_0} & {F_2} \\
	&& {\widehat{x^0_j}} & {\widehat{x^2_j}} && R & \bullet \\
	{\widehat{x^0_i}} & {\widehat{x^2_i}} && {\widehat{r_j}} & \bullet && {F_1} & \bullet \\
	& {\widehat{r_i}} & \bullet && {\widehat{x_j^1}} & \bullet \\
	&& {\widehat{x_i^1}} & \bullet
	\arrow[from=1-5, to=1-6]
	\arrow[two heads, from=1-5, to=2-6]
	\arrow[two heads, from=1-6, to=2-7]
	\arrow[color={rgb,255:red,181;green,181;blue,181}, dotted, no head, from=2-3, to=1-5]
	\arrow[from=2-3, to=2-4]
	\arrow[two heads, from=2-3, to=3-4]
	\arrow[color={rgb,255:red,181;green,181;blue,181}, dotted, no head, from=2-4, to=1-6]
	\arrow[two heads, from=2-4, to=3-5]
	\arrow[from=2-6, to=2-7]
	\arrow[hook, from=2-6, to=3-7]
	\arrow[hook, from=2-7, to=3-8]
	\arrow[color={rgb,255:red,181;green,181;blue,181}, from=3-1, to=2-3]
	\arrow[from=3-1, to=3-2]
	\arrow[two heads, from=3-1, to=4-2]
	\arrow[color={rgb,255:red,181;green,181;blue,181}, from=3-2, to=2-4]
	\arrow[two heads, from=3-2, to=4-3]
	\arrow[from=3-4, to=3-5]
	\arrow[hook, from=3-4, to=4-5]
	\arrow[hook, from=3-5, to=4-6]
	\arrow[from=3-7, to=3-8]
	\arrow[color={rgb,255:red,181;green,181;blue,181}, dashed, from=4-2, to=3-4]
	\arrow[from=4-2, to=4-3]
	\arrow[hook, from=4-2, to=5-3]
	\arrow[color={rgb,255:red,181;green,181;blue,181}, dashed, from=4-3, to=3-5]
	\arrow[hook, from=4-3, to=5-4]
	\arrow[color={rgb,255:red,181;green,181;blue,181}, dotted, no head, from=4-5, to=3-7]
	\arrow[from=4-5, to=4-6]
	\arrow[color={rgb,255:red,181;green,181;blue,181}, from=5-3, to=4-5]
	\arrow[from=5-3, to=5-4]
	\arrow[color={rgb,255:red,181;green,181;blue,181}, dashed, from=5-4, to=4-6]
\end{tikzcd}
}
\]
We get induced maps between the images as well as between the pushouts. Finally note that epis, effective monos and pushout squares grow up to epis, effective monos and pushout squares, respectively.
\end{proof}

\begin{remark}
    $\mathcal{C}$ is $\kappa $-lex, regular. By \cite[Proposition 0.5]{rosicky} in $\mathbf{Lex}_{\kappa }(\mathcal{C},\mathbf{Set})$ the following classes of monomorphisms coincide: effective, regular, strong, extremal.
\end{remark}

In fact, the proper generalization of the finitary case would include that if $\mathcal{C}$ is $\kappa $-regular then so is $Pro_{\kappa }(\mathcal{C})$. 

\begin{theorem}
\label{proiskreg}
    $\mathcal{C}$ is $\kappa $-regular. Then $\mathbf{Lex}_{\kappa }(\mathcal{C},\mathbf{Set})^{op} = Pro_{\kappa }(\mathcal{C})$ is $\kappa $-regular and the Yoneda-embedding $Y:\mathcal{C}\hookrightarrow \mathbf{Lex}_{\kappa }(\mathcal{C},\mathbf{Set})^{op}$ is $\kappa $-regular.
\end{theorem}

\begin{proof}
    $\kappa \geq \aleph _1$ can be assumed.
    
    We have to prove that a $\kappa $-small transfinite composition of regular monos is a regular mono. We know that this holds for chains of representables, and as (sequential) colimits and regular monos both commute with $\kappa $-filtered colimits it suffices to prove that any $\kappa $-small well-ordered chain of regular monos 
    \[
    (F_0\xhookrightarrow{h_0} F_1\hookrightarrow \dots F_{i}\xhookrightarrow{h_i} \dots )_{i<\gamma <\kappa }
    \]
    is the $\kappa $-filtered colimit of such chains of regular monos going between representables.

    By Theorem~\ref{positselski} it is the $\kappa $-filtered colimit of $\gamma $-indexed chains formed among representables (but the maps are arbitrary). ($\kappa $-presentable objects are representables, since $\mathcal{C}$ is Cauchy-complete.) By taking epi - regular mono factorizations we get:

\[
\adjustbox{width=\textwidth}{
\begin{tikzcd}
	{F_0} & {F_0'} & {F_1} & {F_1'} & {F_2} & {\dots } & {F_{\omega }} & {F_{\omega }'} & {F_{\omega +1}} & \dots \\
	{\dots } & {\dots } & {\dots } & {\dots } & \dots && {\dots } & {\dots } & \dots \\
	{\widehat{x_{0,j'}}} & {\widehat{r_{0,j'}}} & {\widehat{x_{1,j'}}} & {\widehat{r_{1,j'}}} & {\widehat{x_{2,j'}}} & {\dots } & {\widehat{x_{\omega ,j'}}} & {\widehat{r_{\omega ,j'}}} & {\widehat{x_{\omega +1,j'}}} & \dots \\
	{\widehat{x_{0,j}}} & {\widehat{r_{0,j}}} & {\widehat{x_{1,j}}} & {\widehat{r_{1,j}}} & {\widehat{x_{2,j}}} & \dots & {\widehat{x_{\omega ,j}}} & {\widehat{r_{\omega ,j}}} & {\widehat{x_{\omega +1,j}}} & \dots
	\arrow["\cong", from=1-1, to=1-2]
	\arrow["{h_0'}", hook, from=1-2, to=1-3]
	\arrow["\cong", from=1-3, to=1-4]
	\arrow["{h_1'}", hook, from=1-4, to=1-5]
	\arrow[from=1-5, to=1-6]
	\arrow["\cong", from=1-7, to=1-8]
	\arrow["{h_{\omega }'}", hook, from=1-8, to=1-9]
	\arrow[from=1-9, to=1-10]
	\arrow[from=3-1, to=2-1]
	\arrow[two heads, from=3-1, to=3-2]
	\arrow[dashed, from=3-2, to=2-2]
	\arrow[hook, from=3-2, to=3-3]
	\arrow[from=3-3, to=2-3]
	\arrow[two heads, from=3-3, to=3-4]
	\arrow[dashed, from=3-4, to=2-4]
	\arrow[hook, from=3-4, to=3-5]
	\arrow[from=3-5, to=2-5]
	\arrow[from=3-5, to=3-6]
	\arrow[from=3-7, to=2-7]
	\arrow[two heads, from=3-7, to=3-8]
	\arrow[dashed, from=3-8, to=2-8]
	\arrow[hook, from=3-8, to=3-9]
	\arrow[from=3-9, to=2-9]
	\arrow[from=3-9, to=3-10]
	\arrow[from=4-1, to=3-1]
	\arrow[two heads, from=4-1, to=4-2]
	\arrow[dashed, from=4-2, to=3-2]
	\arrow[hook, from=4-2, to=4-3]
	\arrow[from=4-3, to=3-3]
	\arrow[two heads, from=4-3, to=4-4]
	\arrow[dashed, from=4-4, to=3-4]
	\arrow[hook, from=4-4, to=4-5]
	\arrow[from=4-5, to=3-5]
	\arrow[from=4-5, to=4-6]
	\arrow[from=4-7, to=3-7]
	\arrow[two heads, from=4-7, to=4-8]
	\arrow[dashed, from=4-8, to=3-8]
	\arrow[hook, from=4-8, to=4-9]
	\arrow[from=4-9, to=3-9]
	\arrow[from=4-9, to=4-10]
\end{tikzcd}
}
\]

We assume that the $\kappa $-filtered category indexing the $\gamma $-sequences going between representables is the canonical one (that is: all such sequences below $(F_i)_i$), and we call it $D$. It is closed under $\kappa $-small colimits (as $\kappa $-small colimits of $\kappa $-presentable objects are $\kappa $-presentable).

Recall that a full subcategory $A\subseteq D$ is cofinal if every $d\in D$ admits an arrow to some $a\in A$. It is closed if it is closed under colimits of $\kappa $-small well-ordered chains (in particular it is closed under isomorphic copies).

We claim that for each $i<\gamma $ those indexes $j\in D$ for which the epi $\widehat{x_{i,j}}\twoheadrightarrow \widehat{r_{i,j}}$ is an iso, form a full cofinal closed subcategory $A_i\subseteq D$. It is clearly closed: any colimit of isos is an iso. It is cofinal: given any $j_0\in D$ the map $\widehat{r_{i,j_0}}\to F_i$ factors through some $\widehat{x_{i,j_1}}$, such that the triangle
\[\begin{tikzcd}
	{\widehat{x_{i,j_1}}} \\
	\\
	{\widehat{x_{i,j_0}}} && {\widehat{r_{i,j_0}}}
	\arrow[from=3-1, to=1-1]
	\arrow[two heads, from=3-1, to=3-3]
	\arrow[from=3-3, to=1-1]
\end{tikzcd}\]
commutes. Similarly we can find $\widehat{x_{i,j_1}}\to \widehat{r_{i,j_2}}$, and the colimit of the maps $\widehat{x_{i,j_n}}\to \widehat{r_{i,j_{n+1}}}$ in the zig-zag is an iso. Since $\kappa $ is uncountable this is a map $\widehat{x_{i,j_{\omega }}}\to \widehat{r_{i,j_{\omega }}}$ in the diagram and $j_{\omega }\in A_i$ is above $j_0$.

It remains to prove that the intersection of $<\kappa $ full cofinal closed subcategories in a $\kappa $-filtered category is full cofinal closed (then we can restrict ourselves to this cofinal subdiagram and by induction on $\gamma $ the theorem follows). The proof is the same as for ordinals.

That is, let $A_i\subseteq D$ be full cofinal closed for $i<\gamma <\kappa $. $\bigcap A_i $ is full and closed, we need that it is cofinal. By induction on $\gamma $ we can assume $A_0\supseteq A_1\supseteq \dots $. Take $d\in D$. Then build a (not necessarily cocontinuous) $\gamma $-chain $d\to x_0\to x_1 \to \dots $ such that $x_i\in A_i$. The colimit of this chain lives in each $A_i$, hence in the intersection and it is above $d$.
\end{proof}

\begin{theorem}
\label{elemiffregmono}
    $\mathcal{C}$ is $\kappa $-lex, regular, $F,G:\mathcal{C}\to \mathbf{Set}$ are $\kappa $-lex. Then $\alpha :F\Rightarrow G$ is elementary iff it is a regular monomorphism in $\mathbf{Lex}_{\kappa }(\mathcal{C},\mathbf{Set})$.
\end{theorem}

\begin{proof}
    By Remark~\ref{keyobserv} $\alpha $ is elementary iff for every mono $i:u\hookrightarrow x$, the square
\[\begin{tikzcd}
	{\widehat{x}} & F \\
	{\widehat{u}} & G
	\arrow[from=1-1, to=1-2]
	\arrow["{\widehat{i}}"', from=1-1, to=2-1]
	\arrow["\alpha", from=1-2, to=2-2]
	\arrow[from=2-1, to=2-2]
\end{tikzcd}\]
    admits a diagonal map making the upper triangle commutative.
    
    But $\widehat{i}$ is epi, hence the lower triangle commutes for free and the lift is unique. That is, $\alpha $ is elementary iff it is right orthogonal to every epimorphism between representables. Since every epimorphism is the $\kappa $-filtered colimit of epimorphisms between representables, and left orthogonality classes are closed under colimits, this is equivalent to $\alpha $ being right orthogonal to every epimorphism, i.e.~to $\alpha $ being a strong mono. But strong = regular.
\end{proof}

\begin{remark}
\label{inpretopmonosarereg}
    Let $\mathcal{C}$ be a regular category in which monomorphisms are regular (e.g.~a pretopos, see \cite[Corollary A1.4.9]{elephant}). We claim that in $\mathbf{Lex}(\mathcal{C},\mathbf{Set})$ every monomorphism is regular. This follows, as in $\mathbf{Lex}(\mathcal{C},\mathbf{Set})$ epimorphisms are filtered colimits of $Y$-images of $\mathcal{C}$-monos, these are regular epimorphisms and those are closed under filtered colimits (as a left class), so in $\mathbf{Lex}(\mathcal{C},\mathbf{Set})$ every epimorphism is regular. Hence epi+mono implies iso, and by the existence of epi - regular mono factorizations we get that all monos are regular.

    From Proposition \ref{atregmonopb} we get that in $\mathbf{Lex}(\mathcal{C},\mathbf{Set})$ every monomorphism is an elementary map. By the above theorem these two arguments prove the same thing as elementary = regular mono.
\end{remark}

\begin{theorem}
\label{lextoreg}
    $\mathcal{C}$ is $\kappa $-regular. Then every $\kappa $-lex functor in $\mathbf{Lex}_{\kappa }(\mathcal{C},\mathbf{Set})$ admits a regular mono\-morphism (i.e.~elementary map) to a $\kappa $-regular functor.
\end{theorem}

\begin{proof}
    For $\kappa =\omega $ this is \cite[Proposition 10, Lecture 15X]{lurie}. We repeat the proof. 

    First note that the transfinite composition of a $\kappa $-indexed chain of effective monos is an effective mono. This follows as we can get the composite as a $\kappa $-indexed sequential colimit in the arrow category, formed by all partial composites, and those are effective monos by Theorem~\ref{proiskreg}. So then the composite is effective mono by \cite[Corollary 1.60]{rosicky}.
\[\begin{tikzcd}
	{F_0} & {F_0} & {F_0} & \dots & {F_0} & {\dots } \\
	{F_0} & {F_1} & {F_2} & {\dots } & {F_{\omega }} & {\dots }
	\arrow[Rightarrow, no head, from=1-1, to=1-2]
	\arrow["1"', from=1-1, to=2-1]
	\arrow[Rightarrow, no head, from=1-2, to=1-3]
	\arrow["{j_0}"', from=1-2, to=2-2]
	\arrow[Rightarrow, no head, from=1-3, to=1-4]
	\arrow["{j_{0,2}=j_1\circ j_0}", from=1-3, to=2-3]
	\arrow[Rightarrow, no head, from=1-5, to=1-6]
	\arrow["{j_{0,\omega }}", from=1-5, to=2-5]
	\arrow["{j_0}", from=2-1, to=2-2]
	\arrow["{j_1}", from=2-2, to=2-3]
	\arrow[from=2-3, to=2-4]
	\arrow[from=2-5, to=2-6]
\end{tikzcd}\]
More generally, a transfinite induction proves that for any ordinal $\gamma $, the transfinite composition of a $\gamma $-sequence of effective monos is an effective mono. The successor case is trivial and in limit steps the previous argument applies.

Now we use the small object argument (\cite[Theorem 2.1.14]{hovey}) with the set: $I$=$\{$regular monos between representables$\}$ (corresponding to the effective epis in $\mathcal{C}$). We get that for any $\kappa $-lex $F$, the terminal map $F\Rightarrow *$ can be factored as $F\Rightarrow F'\Rightarrow *$, such that the first map is a transfinite composition of pushouts of maps from $I$, hence a regular mono, while the second map has the right lifting property against all maps in $I$. It means precisely that $F'$ preserves effective epis.     
\end{proof}

\begin{theorem}
\label{regtocoh}
    $\kappa $ is strongly compact (or $\kappa =\omega $), $\mathcal{C}$ is a $\kappa $-coherent category with $\kappa $-small disjoint coproducts. Then every $\kappa $-regular functor in $\mathbf{Lex}_{\kappa }(\mathcal{C},\mathbf{Set})$ admits a regular monomorphism to a (small) product of $\kappa $-coherent functors.
\end{theorem}

\begin{proof}
    \cite[Theorem 5.14]{bvalued}.
\end{proof}

The following is \cite[Definition 5.1.6]{kashiwara}:

\begin{definition}
    A map is a strict monomorphism if it is the joint equalizer of the class of all parallel pairs which it equalizes.
\end{definition}

\begin{remark}
    In an accessible category a map is a strict monomorphism iff it is the joint equalizer of a set of parallel pairs. I.e.~if $\mathcal{A}$ is $\lambda $-accessible and $f:M\to N $ is a strict monomorphism such that $N$ is $\mu $-presentable for some $\mu \vartriangleright \lambda $ then $f$ is the joint equalizer of the pairs $(g_i,h_i:N\rightrightarrows N_i \ |\ g_if=h_if \text{ and $N_i$ is $\mu $-presentable} )$.
\end{remark}

\begin{remark}
    In a category with pushouts strict = regular = effective (if an arrow equalizes the cokernel pair of $f$ then it equalizes all pairs that $f$ equalizes).  In an accessible category with small products strict = regular (a map equalizes $(g_i,h_i:N\rightrightarrows N_i)_i$ iff it equalizes $\langle g_i \rangle ,\langle h_i \rangle :N\rightrightarrows \prod _i N_i$).
\end{remark}

\begin{theorem}
\label{main}
    $\kappa $ is strongly compact (or $\kappa =\omega$), $(\mathcal{C},E)$ is a $\kappa $-site s.t.~every $E$-family has $<\kappa $ legs, $\mathcal{A}=Mod_{\kappa }(\mathcal{C},E)$ is $\lambda $-accessible for some regular $\lambda \geq \kappa $. Then $\lambda $-pure implies strict mono.
\end{theorem}

\begin{proof}
   The Yoneda-embedding followed by sheafification $\# Y:\mathcal{C}\to Sh(\mathcal{C},\langle E\rangle _{\kappa })$ is $\kappa $-lex, $E$-preserving, and it induces an equivalence between the category of $\mathcal{C}\to \mathbf{Set}$ $\kappa $-lex $E$-preserving and the category of $Sh(\mathcal{C})\to \mathbf{Set}$ $\kappa $-lex cocontinuous functors (given by precomposition and left Kan-extension, see \cite[Theorem 3.15]{presheaftype}). $Sh(\mathcal{C})$ is a $\kappa $-topos, i.e.~an $(\infty ,\kappa )$-coherent Grothendieck-topos. So we can close the image $\# Y[\mathcal{C}]$ under $\kappa $-small disjoint coproducts, effective epi - mono factorizations and $\kappa $-small limits to obtain a (small) $\kappa $-coherent full subcategory $\widetilde{\mathcal{C}}$ with $\kappa $-small disjoint coproducts. Then the full embedding $\widetilde{\mathcal{C}}\hookrightarrow Sh(\mathcal{C})$ induces an equivalence between $Sh(\mathcal{C},\langle E\rangle _{\kappa })$ and $Sh(\widetilde{\mathcal{C}},\{<\kappa \text{ eff.~epic families} \})$, by the comparison lemma \cite[Theorem C2.2.3]{elephant}. Consequently $\mathcal{A}\simeq \mathbf{Coh}_{\kappa }(\widetilde{\mathcal{C}},\mathbf{Set})$.
   
   Let $\alpha :M\Rightarrow N$ be $\lambda $-pure in $\mathbf{Coh}_{\kappa }(\widetilde{\mathcal{C}},\mathbf{Set})$. By Proposition~\ref{purethenelem} it is elementary. By Theorem~\ref{elemiffregmono} it arises as the equalizer of some pair $\beta _1,\beta _2: N\Rightarrow F$ in $\mathbf{Lex}_{\kappa }(\widetilde{\mathcal{C}},\mathbf{Set})$. By Theorem~\ref{lextoreg} and by Theorem~\ref{regtocoh} $F$ admits a (regular) monomorphism to a product of $\kappa $-coherent functors. Write 
\[\begin{tikzcd}
	M && N && {\prod _i N_i}
	\arrow["{\alpha }", Rightarrow, from=1-1, to=1-3]
	\arrow["{\langle \gamma _i \rangle _i}", shift left=2, Rightarrow, from=1-3, to=1-5]
	\arrow["{\langle \delta _i \rangle _i}"', shift right=2, Rightarrow, from=1-3, to=1-5]
\end{tikzcd}\]
   for the resulting diagram. It is an equalizer in $\mathbf{Lex}_{\kappa }(\widetilde{\mathcal{C}},\mathbf{Set})$, i.e.~$\alpha $ is the joint equalizer of the pairs $(\gamma _i, \delta _i)$. But then the same is true in the full subcategory $\mathbf{Coh}_{\kappa }(\widetilde{\mathcal{C}},\mathbf{Set})$.
\end{proof}

We repeat the result in the $\kappa =\omega $ case:

\begin{corollary}
    Let $\mathcal{A}$ be a $\lambda $-accessible category which is axiomatizable in coherent logic, i.e.~there is a coherent category $\mathcal{C}$ s.t.~$\mathcal{A}\simeq \mathbf{Coh}(\mathcal{C},\mathbf{Set})$. Then every $\lambda $-pure map is a strict monomorphism.
\end{corollary}

As another corollary we get:

\begin{theorem}
\label{answer}
    Assume that there is a proper class of strongly compact cardinals. Let $\mathcal{A}$ be a $\lambda $-accessible category. Then there is a regular $\mu \vartriangleright \lambda $ such that $\mu $-pure implies strict mono.
\end{theorem}

\begin{proof}
    Every $\lambda $-accessible category is axiomatizable by some $\lambda $-site $(\mathcal{C},E)$: indeed, by \cite[Theorem 2.58]{rosicky} it can be axiomatized by some sketch whose limit part only contains $\lambda $-small diagrams and by the argument in \cite[Proposition 2.2.7.(1)]{elephant} it can be replaced by a $\lambda $-site. (A cocone is sent to a colimit cocone in $\mathbf{Set}$ iff $a)$ the cocone maps are jointly epimorphic and $b)$ given two elements in the diagram they are identified by the respective cocone maps iff they can be connected by some zig-zag inside the diagram. These are expressible by saying that some families are sent to epimorphic ones.)
    
    Take some $\mu \geq \lambda $ such that $\mu $ is strongly compact, every $E$-family has $<\mu $ legs and $|\mathcal{C}|<\mu $. Then by \cite[Example 2.13.(4)]{rosicky} $\lambda \vartriangleleft \mu $ and by \cite[Theorem 4.7]{presheaftype} $(\mathbf{Lex}_{\lambda }(\mathcal{C},\mathbf{Set})_{<\mu }^{op}, Y[E])$ is a $\mu $-site axiomatizing $\mathcal{A}$. It satisfies the conditions of Theorem~\ref{main}. So $\mu $-pure implies strict mono.
\end{proof}

\begin{remark}
    We must allow $\mu $ to be larger than $\lambda $. In \cite{ADAMEK19961} they construct a small $\omega $-accessible category with an $\omega $-pure map which is not even a strong mono. 
    
    Note that the above theorem trivially holds for small accessible categories: if $\mathcal{A}$ is small then there is some $\mu $ s.t.~$\mu $-pure = $\infty $-pure = split mono. 
\end{remark}

\begin{question}
    Do we need large cardinals for Theorem~\ref{answer}?
\end{question}

\printbibliography

\end{document}